\documentclass[12pt]{amsart}
\usepackage{amssymb, amsmath}

\usepackage{hyperref}
\hypersetup{colorlinks,citecolor=blue,linkcolor=blue}

\newcommand\SL{\operatorname{SL}}
\newcommand\GL{\operatorname{GL}}
\newcommand\SO{\operatorname{SO}}

\newcommand\N{{\mathbb N}}
\newcommand\Z{{\mathbb Z}}
\newcommand\Q{{\mathbb Q}}
\newcommand\R{{\mathbb R}}
\newcommand\C{{\mathbb C}}
\newcommand\A{{\mathbb A}}
\newcommand\F{{\mathbb F}}
\newcommand\cO{{\mathcal O}}

\newcommand\D{\mathcal{D}}

\newcommand\Da{\mathfrak{D}}

\newcommand\G{{\mathrm G}}

\newcommand\Pa{{\mathrm P}}

\newcommand\ru{\mathrm{L}_H^u(x)}
\newcommand\rnu{\mathrm{L}_H^{nu}(x)}
\newcommand\muu{m_H^u(x)}
\newcommand\munu{m_H^{nu}(x)}

\newcommand\Ker{\mathrm{Ker}}

\newcommand\rk{\mathrm{rk}}
\newcommand\Gal{\mathrm{Gal}}

\newcommand\An{\mathrm{A}}
\newcommand\Bn{\mathrm{B}}
\newcommand\Cn{\mathrm{C}}
\newcommand\Dn{\mathrm{D}}
\newcommand\En{\mathrm{E}}
\newcommand\Fn{\mathrm{F_4}}
\newcommand\Gn{\mathrm{G_2}}

\numberwithin{equation}{section}

\newtheorem{theorem}{Theorem}

\newtheorem{prop}{Proposition}[section]
\newtheorem{lemma}[prop]{Lemma}
\newtheorem{thm}[prop]{Theorem}
\newtheorem{cor}[prop]{Corollary}

\newtheorem{conja}{Conjecture}[prop]
\theoremstyle{definition}
\newtheorem{rem}[prop]{Remark}
\newtheorem{example}[prop]{Example}
\newtheorem{step}{Step}

\begin{document}

\title{Counting non-uniform lattices}

\dedicatory{Dedicated to Aner Shalev on his 60th birthday}

\author{Mikhail Belolipetsky}
\thanks{Belolipetsky is partially supported by CNPq and FAPERJ}
\address{
IMPA\\
Estrada Dona Castorina, 110\\
22460-320 Rio de Janeiro, Brazil}
\email{mbel@impa.br}

\author{Alexander Lubotzky}
\thanks{Lubotzky is partially supported by ERC, NSF and ISF}
\address{
Institute of Mathematics\\
Hebrew University\\
Jerusalem 91904\\
ISRAEL}
\email{alex.lubotzky@mail.huji.ac.il}



\begin{abstract}
In \cite{BGLM} and \cite{GLNP} it was conjectured that if $H$ is a simple Lie group of real rank at least $2$, then the number of conjugacy classes of (arithmetic) lattices in $H$ of covolume at most $x$ is $x^{(\gamma(H)+o(1))\log x/\log\log x}$ where $\gamma(H)$ is an explicit constant computable from the (absolute) root system of $H$. In \cite{BL2} we disproved this conjecture. In this paper we prove that for most groups $H$ the conjecture is actually true if we restrict to counting only non-uniform lattices.
\end{abstract}

\maketitle


\section{Introduction}

Let $H$ be a non-compact simple Lie group endowed with a fixed Haar measure
$\mu$ 
and let $X$ be the associated
symmetric space. A classical theorem of Wang \cite{Wa} asserts that if $H$ is
not locally isomorphic to $\SL_2(\R)$ or $\SL_2(\C)$, then for every $0< x
\in\R$ there exist only finitely many Riemannian orbifolds covered by $X$ with
volume at most $x$.
Let $\mathrm{L}_H(x)$ (resp. $\mathrm{AL}_H(x)$) denotes the number of conjugacy
classes of lattices (resp. arithmetic lattices) in $H$ of covolume at most $x$.
By Wang's theorem, if $H$ not locally $\SL_2(\R)$ or $\SL_2(\C)$, 
then $\mathrm{L}_H(x)$ is finite for every $x$. This is also true for
$\mathrm{AL}_H(x)$ even for $H = \SL_2(\R)$ or $\SL_2(\C)$ by a theorem of Borel \cite{Bo2}.

In recent years there has been a growing interest in the asymptotic
behavior of these functions (cf. \cite{BGLM}, \cite{G1}, \cite{G2}, \cite{GLNP}, \cite{B}, \cite{BGLS}, \cite{BL2}, \cite{BLi}, and also \cite{SG} for the positive characteristic case).
Super-exponential upper bounds were given in many cases, bounds which are
optimal at least for the rank one groups $\SO(n,1)$. But for higher rank Lie groups  
the growth rate of the number of lattices in much slower as it is shown in \cite{BL2} and \cite{SG}. 

Recall that if $H$ has $\R$-rank$(H)\ge 2$, then by Margulis's arithmeticity
theorem every lattice $\Gamma$ in $H$ is {\em arithmetic}, i.e. there exists
a number field $k$ with ring of integers $\cO$ and the set of archimedean
valuations $V_\infty$, an absolutely simple, simply connected $k$-group $\G$
and an epimorphism $\phi: \prod_{v\in V_\infty} \G(k_v)^o \to H$, such that
$\Ker(\phi)$ is compact and $\phi(\G(\cO))$ is commensurable with $\Gamma$ (see \cite[Theorem~1, p.~2]{Ma}).
Thus for groups $H$ of real rank at least $2$, we have $\mathrm{L}_H(x)=\mathrm{AL}_H(x)$.
Moreover, Serre conjectured (\cite{S1}) that for all lattices $\Gamma$ in such $H$,
$\Gamma$ has the {\em congruence subgroup property} (CSP), i.e.
$\Ker(\widehat{\G(\cO)} \to \G(\widehat{\cO}))$ is finite in the notations above.
Serre's  conjecture is known by now for most arithmetic groups and, in particular, for all non-uniform lattices (see \cite[Chapter~9]{PR}).
Hence the question of counting (non-uniform) lattices in $H$ boils down to counting
maximal arithmetic groups and their congruence subgroups. Maximal arithmetic subgroups
were counted in \cite{B} and very precise estimates for the number of congruence subgroups in a given
lattice were obtained in \cite{Lu}, \cite{GLP} and \cite{LN}, some of them are conditional
on the generalized Riemann hypothesis (GRH). Even before these results, it was conjectured
in \cite{BGLM} that for groups $H$ of real rank at least $2$ the counting function $\mathrm{L}_H(x)$
grows like $x^{\log x/\log\log x}$.
In fact, a more precise conjecture was made in \cite{GLNP}, where it is suggested that
\begin{equation}
\lim_{x\to\infty} \frac{\log \mathrm{L}_H(x)}{(\log x)^2 / \log\log x} = \gamma(H), \text{ with }
\end{equation}
\begin{equation}\label{eq_11}
\gamma(H) = \frac{(\sqrt{R(R+1)}-R)^2}{4R^2},
\end{equation}
where $R = R(H) = \# \Phi_+/l$ is the ratio of the order of the set of positive roots of the (absolute) root system corresponding to $H$ and the rank of the root system. 

In \cite{BL2} we proved that in general the conjecture is false and the correct rate of growth
is $x^{\log x}$. The purpose of this paper is to show that the conjecture is still true
if one restricts to non-uniform lattices. In a general setting our results
depend on some widely believed number theoretic conjectures, but for ``most'' cases
they are proved unconditionally.

\medskip

For the main theorem of the paper, we assume that if $G$ is a split form of $H$, then the center of the simply connected cover of $G$ is a $2$-group and $G$ has no outer automorphism of order three.
This is the case for most $H$'s. In fact, this assumption says that $H$ is not of type $\En_6$ or $\Dn_4$, and if it is of type $\An_n$, then $n$ is of the form $n = 2^{\alpha}-1$ for some $\alpha\in\N$. We will call such simple Lie groups \emph{$2$-generic}.

\begin{theorem} \label{theorem}
For a $2$-generic simple Lie group $H$ of real rank at least $2$, we have
$$\lim_{x\to\infty} \frac{\log \rnu}{(\log x)^2 / \log\log x} = \gamma(H),$$
where $\gamma(H)$ is as in (\ref{eq_11}) and $\rnu$ is the number of conjugacy classes of non-uniform lattices in $H$ of covolume at most $x$.
\end{theorem}

The theorem, combined with \cite{BL2}, says in particular that there are ``many more'' uniform lattices than non-uniform ones.  This is a characteristic zero phenomenon~--- in positive characteristic the growth of the two types are similar (at least modulo some well established conjectures and independently in ``most'' cases). We refer to \cite{SG} for more on the positive characteristic case. 

The proof of Theorem \ref{theorem} uses Gauss's Theorem which gives us a bound for the $2$-rank of the
class groups of quadratic extensions of $\Q$, the field of rational numbers. In order to be able to extend Theorem \ref{theorem} to all
simple groups $H$ we would need similar bounds for the $l$-ranks of the class groups of number fields for the primes $l \ge 2$. In fact,
we show that it is essentially equivalent to such bounds --- see Section~\ref{sec:equiv conj}.
Unfortunately, with the current knowledge in number theory, such bounds seem to be out of reach.
It is interesting to note that if one could prove Theorem~\ref{theorem} (or even a weaker form of it) by geometric 
methods this would have deep number theoretic consequences.
(See Section~\ref{section:NT} for a further discussion.)

\medskip

The proof of the theorem employs and develops the methods from \cite{B},
\cite{GLP} and \cite{LN} (see also \cite{GLNP}). In particular, we need to show that the error terms in the results that we use from these papers do not depend on the choice of the base lattice. 
We would like to point out that unlike some of the results in \cite{GLP} and \cite{LN}, Theorem~\ref{theorem} does not
depend on the GRH~--- the Generalized Riemann Hypothesis (although the extension to some
semisimple groups will still require such an assumption; see Section~\ref{section:semisimple}).

The lower bound of Theorem \ref{theorem} is proved in Section~\ref{sec:thm2 lower} by using the methods of \cite{GLP} and
\cite{LN}. We can avoid the assumption of the GRH imposed there by choosing suitable lattices for
which Bombieri--Vinogradov ``Riemann hypothesis on the average'' gives results which are as
good for us as the GRH (see Section~\ref{sec:thm2 lower} below).
These methods provide sufficiently many congruence subgroups (with the
appropriate rate of growth) in the chosen arithmetic lattices. We then use
the results of \cite[Section~5]{BL2} to ensure that not too many of them are
conjugate in $H$.

In Section~\ref{sec:thm2 upper} we prove the upper bound of Theorem~\ref{theorem}. This involves estimating the error term in the upper bound in \cite{LN}. Besides new difficulties with the ``subgroup growth'' methods, here one has to overcome the subtle issue with the $l$-rank of the class groups mentioned above. This upper bound is the main novel result of the paper.

In the final Section~\ref{section:semisimple} we discuss the extension of the main results to the non-uniform lattices in semisimple Lie groups and their relation to the number theoretic conjectures which were mentioned above.

\section{Notations and conventions}\label{section:notations}


Let $H$ be a semisimple Lie group without compact factors.
A subgroup $\Gamma$ of $H$ is called a {\em lattice}
if $\Gamma$ is discrete in $H$ and its covolume (with respect to some and hence
any Haar measure on $H$) is finite. A lattice is called {\em irreducible} if $\Gamma N$
is dense in $H$ for every non-compact closed normal subgroup $N$ of $H$.
A lattice is called {\em uniform} (resp.
{\em non-uniform}) if $H/\Gamma$ is compact (resp. non-compact).

Two subgroups $\Gamma_1$ and $\Gamma_2$ are called {\it commensurable} if
$\Gamma_1\cap\Gamma_2$ is of finite index in both of them. If $\Gamma$ is a
lattice in $H$, its {\it commensurability subgroup} (or {\it commensurator}) in
$H$ is defined as
\begin{equation*}
{\rm Comm}_H(\Gamma) = \{ g\in H \mid g^{-1}\Gamma g\ {\rm \ \and\ } \Gamma
{\rm\ are\ commensurable}\}.
\end{equation*}

For a finite group (resp. profinite group) $G$ we define its {\em rank} $\rk(G)$ as the supremum of the minimal number of generators over the subgroups (resp. open subgroups) of $G$. If $p$ is a prime and $G$ is a finite group, the \emph{$p$-rank of $G$} is defined by $\rk_p(G) = \rk(P)$, where $P$ is a Sylow $p$-subgroup of $G$.

\medskip

Along the lines we shall often come to arithmetic considerations, for which we
now fix some notations. Throughout this paper $k$ will always denote a number
field, $\cO = \cO_k$ is its ring of integers, $\D_k$ is the absolute value of the
discriminant $\Delta_k$ and $h_k$ is the class number of $k$. The set of valuations
(places) of $k$, $V = V(k)$, is the union of the set $V_\infty$ of archimedean
and the set $V_f$ of nonarchimedean (finite) places of $k$.
The number of archimedean places of $k$ is denoted by $a = a_k = \# V_\infty$,
and $r_1$, $r_2$ denote the number of real and complex places of $k$, respectively
(so $a = r_1 + r_2$ and $d = d_k = [k:\Q] = r_1 + 2r_2$).
Given a nonarchimedean place $v\in V_f$, the
completion of $k$ with respect to $v$ is a nonarchimedean local field $k_v$, its residue
field, which will be denoted by $\F_v$ or $\F_q$, is a finite field of cardinality
$q = q_v$. Finally, $\A = \A(k) = \prod_{v\in V}'k_v$ is the ring of ad\`eles of $k$,
where $\prod'$ denotes a restricted product.

\medskip

All logarithms in this paper will be taken in base $2$. For a real number $x$, $[x]$ denotes the largest integer $\le x$. The number of elements of a finite set $S$ will be denoted by $\# S$, while the order of a finite group $G$ will be denoted by $|G|$.

\medskip

Whenever it is not stated otherwise, the constants $c_1$, $c_2$, etc. depend only on the Lie group $H$.

\section{Number theoretic background}\label{section:NT}

Our problem of counting arithmetic lattices turns out to be intimately connected with estimating the rank of class groups. The reader is encouraged to take a look at Proposition~\ref{lemma_Cl_n} to get a first hint on the connection. So let us start here with some results and conjectures on class groups.

Let $h_k$ be the class number of the number field $k$, which is defined as the order of the class group $\mathrm{Cl}(k)$. Let $\mathrm{Cl}_m(k)$ denote the $m$-torsion subgroup of $\mathrm{Cl}(k)$, which consists of the elements of the class group whose orders divide $m$. We denote by $h_{k,m}$ the order of $\mathrm{Cl}_m(k)$, so $h_{k,l} = l^{\mathrm{rk}_l(\mathrm{Cl}(k))}$ when $l$ is a prime.

\begin{thm}\label{thm:Gauss}{\rm (Gauss, see \cite[Theorem 8, p. 247]{BS})}
Let $k$ be a quadratic number field with discriminant $\Delta_k$. Suppose that $\Delta_k$ has $t$ distinct prime divisors. Then
$$h_{k,2} = \left\{\begin{array}{ll}
2^{t-2} & \text{if $\Delta_k > 0$, and $p|\Delta_k$ for some prime $p\equiv 3\;(\mathrm{mod}\ 4)$};\\
2^{t-1} & \text{otherwise.}
\end{array} \right.$$
\end{thm}

In fact, we do not need the exact formulation of Gauss's result but rather its corollary:
\begin{cor}\label{cor:Gauss}
For $k$ running over the quadratic number fields, we have
$$\rk_2(\mathrm{Cl}(k)) = O\left(\frac{\log\D_k}{\log\log\D_k}\right).$$
\end{cor}

\noindent
The corollary follows from the theorem by recalling the prime number theorem, which implies that the number of different prime divisors of $\D_k$ is $O(\log\D_k/\log\log\D_k)$.

Gauss proved his celebrated theorem using the genus theory which he developed. Later on the method was applied to prove a similar result for biquadratic extensions of $\Q$ (see \cite[Theorem~2]{Co}). We use these results in the proof of Theorem~\ref{theorem}; the main restriction on the type of $H$ which we impose there follows from the fact that in order for the genus theory to be applicable we need to assume that the center of the simply connected cover of the split form of $H$ is a $2$-group. In order to extend the theorem to other types we will require somewhat similar estimates for $\rk_l(\mathrm{Cl}(k))$ for $l>2$.

Apparently, even though related questions were extensively studied, very little is known about the $l$-ranks, for $l\neq 2$, of class groups of quadratic extensions and the ranks of class groups of extensions of higher degree.
We refer to \cite{EV} for some recent results and a review of the status of the problem. The ``folk'' conjecture which appeared in several sources states that for a fixed prime $l$ the $l$-part of the class group of number fields $k$ of a fixed degree grows slower than any power of $\D_k$ (see \cite[Section~1.2]{EV}). What we need is a stronger statement:

\addtocounter{prop}{1}
\renewcommand{\theconja}{\theprop.\Alph{conja}}

\begin{conja}\label{conj_h}
Fix an integer $d\ge 2$ and a prime $l$. Then for number fields $k$ of degree $d$, $\rk_l(\mathrm{Cl}(k)) =
o(\frac{\log\D_k}{\sqrt{\log\log\D_k}})$.
\end{conja}

\noindent For a future reference we will also formulate a slightly weaker conjecture:

\begin{conja}\label{conj_h'}
In the notations of Conjecture \ref{conj_h},
$\rk_l(\mathrm{Cl}(k)) = O(\frac{\log\D_k}{\sqrt{\log\log\D_k}})$.
\end{conja}

Let us note that a closely related statement to Conjecture~\ref{conj_h'} appears as a question in \cite[p.~96]{BrS}: Brumer and Silverman asked if there exists a constant $c_{l,d}$ depending only on $d$ and $l$ such that 
$\rk_l(\mathrm{Cl}(k)) \le c_{l,d} \frac{\log\D_k}{\log\log\D_k}$.

When we will prove Theorem~\ref{theorem} in Section~\ref{sec:thm2 upper}, we will use Corollary~\ref{cor:Gauss} of Gauss's theorem but a slightly weaker estimate as in Conjecture~\ref{conj_h} would suffice.
We will see in Section~\ref{sec:equiv conj} that proving Conjecture~\ref{conj_h} for quadratic and biquadratic fields would imply Theorem~\ref{theorem} for all simple Lie groups, and proving Theorem~\ref{theorem} for all simple groups would imply Conjecture~\ref{conj_h'}  for quadratic imaginary fields and some biquadratic fields. Proving it for all semisimple groups (with non-uniform lattices) would imply Conjecture~\ref{conj_h'} for all number fields $k$ and primes $l > 2$ (see Section~\ref{sec:semi_upper_thm2}). So, these results give some support for the conjectures.

The $l$-ranks of the class groups and, in particular, the $l$-ranks of the class groups of quadratic number fields are subject of the work of Cohen and Lenstra \cite{CL}. Heuristic assumptions introduced in \cite{CL} allowed them to make striking predictions about the average values of the ranks. Let us mention that although our conjectures refer to the same object, they neither imply nor follow from the Cohen--Lenstra heuristics ---  the latter deal with the averages while our conjectures refer to the extreme values of the ranks of the class groups.

\section{Arithmetic subgroups}\label{sec:ar}

Let $H$ be a semisimple connected linear Lie group without compact factors. It
is known that if $H$ contains irreducible lattices then all of its
simple factors are of the same type. Such groups $H$ are called {\it isotypic} or {\it
typewise homogeneous} (see~\cite[Chapter~9.4]{Ma}). So from now on we shall assume
that $H$ is isotypic. Moreover, without loss of generality we can further assume
that the center of $H$ is trivial. This implies that $H$ is isomorphic to ${\rm
Ad\:} H$, where ${\rm Ad}$ denotes, as usual, the adjoint representation. The
group ${\rm Ad\:} H$ is the connected component of identity of the $\R$-points
of a semisimple algebraic $\R$-group. There exist, therefore,
absolutely simple $\R$-groups $\G_i$, all of the same type, such that $H = (\prod_{i=
1}^{a}\G_i(\R) \times \G_{a+1}(\C)^{b})^o.$ A classical theorem of Borel
\cite{Bo1} (see also \cite{BH}) asserts that such $H$ does contain irreducible
lattices.

A word of warning: there are cases in which $H$ contains uniform irreducible
lattices but has no such a non-uniform lattice. If $H$ is simple non-compact
it has both uniform and non-uniform arithmetic lattices. We refer to \cite[Chapter~18.7]{WM}
for a discussion of this issue.

Let now $\G$ be an algebraic group defined over a number field $k$ which admits
an epimorphism $\phi:\G(k\otimes_\Q\R)^o \to H$ whose kernel is compact. In
this case, $\phi(\G(\cO))$ is an irreducible lattice in $H$.
Such lattices and the subgroups of $H$ which are
commensurable with them are called {\it arithmetic}. It can be shown that to
define all irreducible arithmetic lattices in $H$ it is enough to consider only
simply connected, absolutely almost simple $k$-groups $\G$ which have the same
(absolute) type as the almost simple factors of $H$ and are defined over the
fields with at most $b$ complex and at least $a$ real places. We shall
call such groups $\G$ and corresponding fields $k$ {\it admissible}.

Let us recall now a fact which is crucial for this paper and explains why the results here are so different from those in \cite{BL2}: If $\Gamma$ is a non-uniform arithmetic lattice in $H$, then its field of definition $k$ is of bounded degree over $\Q$ with a bound depending only on $H$. This follows from the well known result that the quotient space $H/\Gamma$ is non-compact if and only if $\Gamma$ contains non-trivial unipotent elements (see \cite[Chapter~5.3]{WM}) and hence non of the factors of $\G(k\otimes_\Q\R)$ is compact, which implies that the number of archimedean completions of $k$ is bounded by the number of simple factors of $H$ and so $[k:\Q]$ is bounded.

The arithmetic subgroups of the semisimple Lie group $H$ can be also described by the following construction which we are going to use throughout the paper.
Let $\Pa = (\Pa_v)_{v\in V_f}$ be a collection of parahoric subgroups $\Pa_v\subset\G(k_v)$
of a simply connected $k$-group $\G$. The family $\Pa$ is called {\it coherent} if
$\prod_{v\in V_\infty}\G(k_v)\cdot\prod_{v\in V_f} \Pa_v$ is an open subgroup of
the ad\`ele group $\G(\A_k)$.
Now let
\begin{equation*}
\Lambda = \Lambda(\Pa) = \G(k)\cap\prod_{v\in V_f} \Pa_v,
\end{equation*}
where $\Pa$ is a coherent collection. Following \cite{Pr}, we shall call $\Lambda$ the
{\em principal arithmetic subgroup} associated to $\Pa$. We shall also call $\Lambda' = \phi(\Lambda)$
a principal arithmetic subgroup of $H$.

Let now $\Gamma$ be a maximal arithmetic lattice in $H$. It is known that $\Gamma$ can be obtained as a normalizer in $H$ of the image $\Lambda'$ of some principal arithmetic subgroup $\Lambda$ of $\G(k)$ (see \cite[Proposition~1.4(iv)]{BP}). Moreover, such $\Lambda$ is a principal arithmetic subgroup of {\em maximal type} in a sense of Rohlfs (see \cite{Rohlfs} and also \cite{CR} for precise definitions). In order to prove the main theorem we will need certain control over the structure of $\Lambda$ and the index $[\Gamma:\Lambda']$ in terms of the covolume of $\Gamma$. For this purpose we recall some results from \cite{B}. As it is explained there, the group $\Gamma/\Lambda'$ is always abelian and the prime divisors of its order are contained in a finite set which depends only on $H$.

Let $\Gamma = N_H(\Lambda')$, where $\Lambda' = \phi(\Lambda)$, $\Lambda = \G(k)\cap\prod_{v\in V_f} \Pa_v$, be a maximal arithmetic lattice of covolume less than $x$, with $x$ large enough.

\begin{prop} \label{cor:B} \cite[Corollaries~6.1, 6.3]{B}
There exists a constant $C = C(H)$ such that for $Q = \Gamma/\Lambda'$ we have:
\begin{itemize}
\item[(i)] $|Q| \le x^{C}$;
\item[(ii)] If $\Gamma$ is non-uniform and $H$ is $2$-generic,
then $|Q| \le C^{\log x / \log\log x}$.
\end{itemize}
\end{prop}

\begin{rem} \label{rem_43}
The groups that we excluded in Part~(ii) of the proposition are those for which the center of the simply connected form is not a $2$-group or $\Dn_4$, the last would require considering the class groups of cubic fields.
If we could prove Conjecture~\ref{conj_h}, then the argument from \cite{B} would allow us to prove a slightly weaker form of Proposition~\ref{cor:B}(ii) for non-uniform lattices in an arbitrary semisimple group $H$. More precisely, it would imply that $|Q| = C^{o\left(\log x/\sqrt{\log\log x}\right)}$ which is sufficient for our purpose.
\end{rem}

We will need a variant of the "level versus index" lemma where the level is controlled by the covolume of the lattice. To put it in a perspective, recall the classical lemma asserting that in $\Delta = \SL_2(\Z)$, every congruence subgroup of index $n$ contains $\Delta(m) = \Ker(\SL_2(\Z)\to\SL_2(\Z/m\Z))$ for some $m \le n$, i.e. the level $m$ is at most the index $n$. This was generalized in \cite{Lu} to the congruence subgroups of an arbitrary arithmetic group $\Delta$ by paying a price for $m$; i.e. it was shown that $m\le Cn$ for some constant $C$ which depends on the arithmetic group $\Delta$. Here we want to bound $C$ in terms of the covolume of $\Delta$.

Let us first introduce some notations. As before, let $\Lambda = \G(k)\cap\prod_{v\in V_f}\Pa_v$
where $k$ is a number field with the ring of integers $\cO$, $\G$ is a $k$-form of $H$ and $\Pa_v$ is a parahoric subgroup of $\G(k_v)$, and  let $\G_v$ be an $\cO_v$-scheme  with the
generic fiber isomorphic to $\G(k_v)$ such that $\G_v(\cO_v) = \Pa_v$. This
induces a congruence subgroup structure on $\Pa_v$ defined as follows:
\begin{equation*}
\Pa_v(r) = \Ker (\G_v(\cO_v) \to \G_v(\cO_v/\pi_v^r\cO_v)),
\end{equation*}
where $\pi_v$ is a uniformizer of $\cO_v$. These congruence subgroups induce a
congruence structure on $\Lambda$, $\Lambda(\pi_v^r) = \Pa_v(r)\cap\Lambda$. More
generally, for every ideal $I$ of $\cO$ look at its closure $\bar I$ in
$\hat\cO = \prod_v \cO_v$. Then $\bar I$ is equal to $\prod_{i=1}^{l}
\pi_{v_i}^{e_i}\hat\cO$ for some $Y = \{v_1, \ldots, v_l\}\subset V_f$ and
$e_1,\ldots ,e_l \in \N$. We then define the $I$-congruence subgroup of $\Lambda$,
\begin{equation*}
\Lambda(I) = \Lambda \cap (\prod_{i=1}^{l} \Pa_{v_i}(e_i)\cdot\prod_{v\not\in Y}\Pa_v).
\end{equation*}
In particular, for every $m\in\N$, the $m$-congruence subgroup $\Lambda(m) = \Lambda(m\cO)$ is
defined. Any subgroup of $\Lambda$ which contains $\Lambda(I)$ for some non-zero ideal $I$ is
called a {\em congruence subgroup}.

Let now $\Lambda$ be a principal arithmetic subgroup of a maximal type in $\G(k)$ and let $\Lambda'$ be its image in $H$. Assume also that
$\mu(H/\Lambda') \le x$, where $x\gg 0$. We have the following effective level versus index lemma proved in \cite{BL2}:
\begin{lemma}\label{level vs index lemma} \cite[Lemma~4.3]{BL2}
If $\Lambda_1$ is a congruence subgroup of $\Lambda$ of index $n$, then $\Lambda_1 \supseteq \Lambda(m\cO)$ where $m\in\N$ with $m\le x^C n$ and $C$ is a constant which depends only on $H$.
\end{lemma}

Note that for a general lattice $\Lambda'$ of $H$, obtained from a principal arithmetic subgroup $\Lambda$ as before, the index of $m\cO$ in $\cO$ (and hence also of $\Lambda(m\cO)$ in $\Lambda$) is not necessarily polynomial in $m$. It is bounded by $m^d$, where $d$ is the degree of the defining field $k$ of the arithmetic subgroup $\Lambda$. As it is explained in \cite{BL2}, in general the degree $d$ is bounded by $O(\log x)$, and hence the index of $\Lambda(m\cO)$ in $\Lambda$ is bounded by $(xn)^{c\log x}$.
A better result is probably true: $\Lambda_1 \supseteq \Lambda(I)$ for some $I\lhd\cO$ such that $[\Lambda:\Lambda(I)]\le (xn)^c$ with a constant $c$ depending only on $H$. This indeed follows from Lemma~\ref{level vs index lemma} if the degree of the field $k$ is bounded, which is always the case for non-uniform lattices:

\begin{cor} \label{level vs index}
In the notation above, if $\Lambda'$ is a principal non-uniform arithmetic lattice in $H$ of covolume at most $x$ and $\Lambda_1$ is a congruence subgroup of $\Lambda$ of index $n$, then $\Lambda_1 \supseteq \Lambda(m\cO)$ with $m\le x^{c_1} n$ and $[\Lambda:\Lambda(m\cO)]\le (xn)^{c_2}$, where the constants $c_1$, $c_2$ depend only on $H$.
\end{cor}



\section{Proof of the lower bound} \label{sec:thm2 lower}

The proof of the lower bound of Theorem \ref{theorem} depends on the lower bound of the following result:
\begin{thm}\label{thm:GLP} {\rm (\cite{GLP}, \cite{LN})}
Let $k$ be a number field, $\G$ an absolutely simple, connected, simply connected $k$-group with a fixed $k$-embedding $\G\subset\GL_s$ and let $\Gamma = \G(k)\cap\GL_s(\cO)$, where $\cO$ is the ring of integers of $k$. Denote by $c_n(\Gamma)$ the number of congruence subgroups of $\Gamma$ of index at most $n$.

Then, assuming the Generalized Riemann Hypothesis (GRH),
\begin{equation*}
\lim_{n\to\infty} \frac{\log c_n(\Gamma)}{(\log n)^2/\log\log n} = \gamma(\G),
\end{equation*}
where $\gamma(\G) = (\sqrt{R(R+1)}-R)^2 / 4R^2$ with $R = R(\G) = \# \Phi_+ /
r$, $\Phi_+$ is the set of positive roots of $\G$ and $r$ is its absolute rank.

The same result holds unconditionally if the quasi-split inner form of $\G$ splits over some abelian extension of $\Q$.
\end{thm}

\begin{rem} \label{rem_52}
For our purposes we need also the following observation concerning the
proof in \cite[Section~4]{GLP}: The proof of the lower bound provides the required
number of congruence subgroups of index at most $n$ in the preimage of $\mathrm{B}(P) =
\prod_{p\in S} \mathrm{B}(p)$ in $\Gamma = \G(\cO)$, where $S$ is a carefully chosen set
of rational primes which split completely in $\cO$, $\pi(p)$ is a maximal ideal
of $\cO$ which lies over $p$, $P = \prod_{p\in S}\pi(p)$ and $\mathrm{B}(p)$ is a
Borel subgroup of $\G(\cO/\pi(p)) = \G(\Z/p\Z)$ (see \cite[p.~87]{GLP} for details).
Moreover, $[\cO:P]$ is bounded polynomially by $n$ and so is the order of $\G(\cO/P)$.
As all these index $n$ subgroups contain the principal congruence subgroup
$\Gamma(P)$, we will be able to apply to them later Corollary~5.3 from \cite{BL2},
which implies that the number of such subgroups which are mutually conjugate in $H$
is bounded above by a polynomial of $n$.
\end{rem}

Another result which we shall use is:

\begin{prop}\label{thm:WM} {\rm (\cite[Proposition~6.1]{M1}, \cite[Proposition~3]{PR06})}
Let $\G_0$ be an absolutely simple, simply connected algebraic $\R$-group. Then
there exists a $\Q$-group $\G$ such that:
\begin{itemize}
\item[(1)] $\G \cong \G_0$ over $\R$;
\item[(2)] $\G$ splits over $\Q[i]$;
\item[(3)] $\G$ is quasi-split over $\Q_p$, for every odd prime $p$;
\item[(4)] $\Q$-{\rm rank}$(\G)$ = $\R$-{\rm rank}$(\G_0)$.
\end{itemize}
\end{prop}

Let $H$ be a non-compact simple Lie group, so $H = {\mathrm H_1}(\R)^o$ for some simple algebraic group ${\mathrm H_1}$ defined over $\R$. If ${\mathrm H_1}$ is absolutely simple, we call the group $H$ {\em real}, and otherwise we call it {\em complex}. In the second case, up to isogeny, $H = {\mathrm H}_2(\C)$ and ${\mathrm H_1} = \mathrm{Res}_{\C/\R} {\mathrm H}_2$ for a complex group ${\mathrm H}_2$.

Assume first that $H$ is a real simple Lie group. Then (up to the center and connected component)
$H$ is equal to $\G(\R)$ for some $\Q$-group $\G$ as in Proposition~\ref{thm:WM}.
If $\G$ splits over $\Q$ then, since $H$ is a simple Lie group, $\G$ is absolutely
simple. If $\G$ does not split over $\Q$, it is a restriction of scalars from $l = \Q[i]$ to $\Q$ of
an absolutely simple split $l$-group $\G_1$. In either case we can apply
Theorem~\ref{thm:GLP} unconditionally of the GRH by taking $\Gamma = \G(\Z)$:
In the first case $\G$ is absolutely simple, and in the second case
$\G(\Z)=\G_1(\Z[i])$, where $\G_1$ is absolutely simple and $\Q[i]$ is an abelian
extension. We conclude that
\begin{equation*}
\liminf_{n\to\infty} \frac{\log c_n(\Gamma)}{(\log n)^2/\log\log n} \ge
\gamma(\G).
\end{equation*}
A crucial point for us is that $\gamma(\G)$ depends only on the absolute type
of $\G$ which is completely determined by $H$.

If $H$ is complex, then we can take $\G$ to be a split $\Q$-group such that $\G(\C) = H$. We automatically obtain $\Q$-{\rm rank}$(\G)$ = $\R$-{\rm rank}$(H)$. As $\G$ is absolutely simple, we can take $\Gamma = \G(\Z[i])$, which is a lattice in $\G(\C) = H$, and again Theorem \ref{thm:GLP} applies unconditionally to $\Gamma$.

To finish the proof of the lower bound of Theorem~\ref{theorem},
we now treat both cases together. The subgroup $\Gamma$ is a principal arithmetic subgroup which defines a lattice in $H$
which we denote by the same letter. This lattice is non-uniform since $\Q$-rank$(\G)$ = $\R$-rank$(H) > 0$,
as $H$ is non-compact (see \cite[Chapter~2]{WM}). Let $\nu_0 = \mu(H/\Gamma)$, so every
index $n$ subgroup of $\Gamma$ gives a lattice of covolume $x = n\nu_0$ in $H$.
Remark~\ref{rem_52} and \cite[Corollary~5.3]{BL2} are now combined to show that
among the index $n$ subgroups of $\Gamma$, which give the lower bound of
Theorem~\ref{thm:GLP}, only polynomially many (in $n$ or, equivalently, in $x$)
are in the same $\overline{\mathrm{G}}(k)$-conjugacy class, where $\overline{\mathrm{G}}$
is the adjoint form of $\G$ (if needed, one can
replace $\Gamma$ by a commensurable subgroup which satisfies the assumptions
of Proposition~5.2 from \cite{BL2}, i.e., all the $\Pa_v$ associated with it are maximal; we can then
apply Corollary~5.3 of \cite{BL2} for our case). By the discussion in \cite[Section~5]{BL2} it
follows that the same lower bound
applies also when we count the lattices up to $H$-conjugacy, and the lower bound of
Theorem \ref{theorem} is proven. \qed

\section{Proof of the upper bound} \label{sec:thm2 upper}

As in \cite{BL2}, the proof of the upper bound makes an essential use of two ingredients:
\begin{itemize}
\item[(a)] Counting maximal arithmetic lattices (\cite{B} which in turn is
greatly influenced by \cite{BP}); and
\item[(b)] Counting congruence subgroups of arithmetic groups (\cite{LN} and
\cite{GLNP}).
\end{itemize}
The main challenge here is that the bound in Theorem~\ref{theorem} requires finer counting than in \cite{BL2},
where the upper bound grows much faster.

Let us recall the main result of \cite{B} which we are going to use in this section (see also \cite[Theorem~1.6]{BGLS} for the semisimple groups of type $\An_1$):

\begin{thm} \label{thm:B}
Let $H$ be a semisimple Lie group of real rank $\ge2$ without compact factors. Denote by $\muu$ (resp. $\munu$)
the number of conjugacy classes of maximal irreducible uniform (resp. non-uniform) lattices in $H$ of covolume at most $x$. Then:
\begin{itemize}
\item[(i)] For every $\epsilon > 0$, there exists $c = c(\epsilon, H)$ such
that $\muu \le x^{c(\log x)^\epsilon}$ for every $x\gg 0$.
\item[(ii)] There exists a constant $c = c(H)$ such that $\munu \le x^c$ for
every $x\gg 0$.
\end{itemize}
\end{thm}


We shall count the lattices of covolume at most $x$ by first counting the maximal ones (the number of which is small by Theorem~\ref{thm:B}),
and then counting finite index subgroups within such maximal lattices. We will divide the proof into several steps:

%

\begin{step}
Since the result of the main theorem does not depend on the normalization of the Haar measure $\mu$ on $H$, for the sake of convenience we will fix $\mu$ so that $\mu(H/\Gamma) \ge 1$ for every lattice $\Gamma$. This is possible by Wang's theorem mentioned in the introduction (in fact, by an earlier result of Kazhdan--Margulis, giving a positive lower bound for the covolumes of lattices in $H$ --- see \cite[Chapter XI]{Rag}).

The proof of the upper bound in Theorem \ref{theorem} starts in a similar way to the proof of the upper bound in \cite{BL2} and we shall use the same notations as there. There is still a crucial difference: if $\Gamma$ is a non-uniform arithmetic lattice in $H$ then, as explained in Section~\ref{sec:ar}, its field of definition $k$ is of bounded degree over $\Q$. This is the reason why we get a lower rate of growth. This is so for every semisimple group $H$. If $H$ is simple as we assume here, we even have $[k:\Q]\le 2$.
\end{step}

\begin{step}
In analogous to the counting function $\ru$ considered in \cite[Section~7]{BL2}, we have
\begin{equation}\label{eq_71}
\rnu \le \munu \cdot \sup_{\substack{\Gamma \\ \mu(H/\Gamma)\le x}} \!\!\!\! s_x(\Gamma),
\end{equation}
where $\Gamma$ now runs over the maximal non-uniform lattices in $H$ (the number of which grows polynomially by Theorem~\ref{thm:B}(ii)) and $s_x(\Gamma)$ denotes the number of subgroups of $\Gamma$ of index at most $x$. So our main goal in the rest of the proof is to bound the second term of the right hand side of inequality \eqref{eq_71}. 
\end{step}

\begin{step}
Let $\Gamma = N_H(\Lambda')$, $\Lambda' = \phi(\Lambda)$, where $\Lambda = \G(\cO)$ and $\cO$ is the ring of integers in a number field $k$ (see Section~\ref{sec:ar}). We endow $\Lambda'$ with the congruence structure induced from $\Lambda$. As $\Lambda'$ is non-uniform, $\Lambda$ satisfies the congruence subgroup property, so we need to count only congruence subgroups. To be precise, one should say that $\Lambda$ satisfies the weak congruence subgroup property, i.e. $\mathrm{C}(\Lambda) = \mathrm{Ker}(\widehat{\G(\cO)} \to \G(\widehat{\cO}))$ is finite. So, if this kernel is non-trivial there are also some non-congruence subgroups. But we also know from the precise calculation of this congruence kernel in \cite{PR96} that if $[k:\Q]$ is bounded, then $|\mathrm{C}(\Lambda)|$ is bounded. This implies that every finite index subgroup $\Delta < \Lambda$ is of bounded index in its congruence closure $\overline{\Delta}$, and the map $\Delta \to \overline{\Delta}$ has bounded fibers. Hence for the sake of proving Theorem~\ref{theorem}, it suffices to count only the congruence subgroups, as we will do in the remaining part of this section. 
\end{step}

\begin{step}
Using Corollary~\ref{level vs index} (``level versus index'') we can estimate $s_x(\Lambda')$ from above:
$$ s_x(\Lambda') \le \sum_{m=1}^{[x^{c_1+1}]} s_x(\Lambda/\Lambda(m)).$$
\end{step}

\begin{step}
Now, we recall some results on counting congruence subgroups from \cite{GLP} and \cite{LN}. Consider the exact sequence
\begin{equation}\label{eq_72}
1\to N \to \Lambda/\Lambda(m) \to \Lambda/\Lambda(\Pi_{i=1}^{t}p_i) \to 1,
\end{equation}
where $p_1$, \dots, $p_t$ are the different prime divisors of $m$.

By \cite[Corollary~5.2]{GLP},
\begin{equation}\label{eq_73}
s_x(\Lambda/\Lambda(m)) \le m^{c\log\log m} \cdot s_x(\Lambda/\Lambda(\Pi_{i=1}^{t}p_i)),
\end{equation}
where $c$ is a constant which does not depend on $\Lambda$ but only on $H$. We remark that in \cite{GLP} the group $\G$ was assumed to be split, but this assumption was not really needed for the proof of Corollary~5.2 which is based on Lemma~5.1 there, and the only critical issue there is that $N$ is solvable (in fact, even nilpotent) group. 

Thus, the crucial factor to bound is $s_x(\Lambda/\Lambda(\prod_{i=1}^{t}p_i))$. 
\end{step}

\begin{step}
From this point onward we can follow the proof of the upper bound of Theorem~2(A) in~\cite{LN}
showing that $s_x(\Lambda) \le x^{(\gamma(H)+l(x))\log x / \log\log x}$, where $l(x) = o(1)$.
The proof of Theorem~2(A) there is given for a fixed $k$ and a fixed arithmetic group $\G(\cO)$. 
\end{step}

\begin{step}
\emph{We claim that when $k$ is a number field of bounded degree and $\G(\cO)$ runs over principal arithmetic lattices in a fixed group $H$, the result is still valid with the same $l(x)$ for all these lattices. }

The proof in \cite{LN} is long and quite complicated, so we will not repeat it here. Let us only outline the main strategy, elaborating on one point which needs a careful discussion to ensure that the error term does not depend on $\G(\cO)$.
 
The proof on pages 110--122 of \cite{LN} applies the Larsen--Pink theorem and a detailed analysis of the structure of parabolic subgroups of finite simple groups to deduce that the main contribution to the subgroup counting of $\Lambda/\Lambda(\Pi_{i=1}^{t}p_i) = \Pi_{i=1}^{t} \G(q_i)$ comes from the abelian sections $\Pi_{i=1}^{t} \mathrm{B}(q_i)/\mathrm{U}(q_i)$, where $\mathrm{B}$ is a Borel subgroup and $\mathrm{U}$ is its unipotent radical. In fact, we can even concentrate on the cases where $\G(q_i)$ splits. In this case $\mathrm{B}(q_i)$ is of index $q_i^{\#\Phi_+}$ and the abelian group $\mathrm{B}(q_i)/\mathrm{U}(q_i)$ is of order $(q_i-1)^{\mathrm{rank}(\G)}$. (This is the reason why $R(\G) = \frac{\#\Phi_+}{\mathrm{rank}(\G)}$ plays such an important role and is the basis for the computation of $\gamma(\G)$.)

There is however one point which requires a careful discussion when adapting the proof in \cite{LN} to our ``uniform'' result: In \cite{LN}, one works with one group $\G(\cO)$ at a time, so it was possible to ignore finitely many ``bad primes'' which have no importance for the result there but we should check that they do not affect the uniform upper bound. Recall that to each nonarchimedean place $v \in V_f(k)$ we have an associated parahoric subgroup $\Pa_v \subset \G(k_v)$ (see Section~\ref{sec:ar}). Similar to \cite[Section~4.4]{B}, there are three types of bad primes which correspond to the places $v \in V_f(k)$ that satisfy one of the following conditions:
\begin{itemize}
\item[(i)] the parahoric subgroup $\Pa_v$ is not hyperspecial and $\G$ splits over the maximal unramified extension $\hat{k}_v$ of $k_v$; 
\item[(ii)] the group $\G$ is not quasi-split over $k_v$ and splits over $\hat{k}_v$; 
\item[(iii)] the quasi-split inner form of $\G$ is not split over $\hat{k}_v$. 
\end{itemize}
By \cite[Section~4.4]{B}, the total number of bad primes is bounded by $\log(x)$ but this bound is not sufficient to ensure, by a naive counting argument, that these primes are insignificant and a more careful analysis is needed. 
\end{step}

Let us first consider an illustrative example: 

\begin{example}\label{exmpl}
Let $\Da$ be a $\Q$-defined quaternion algebra and $\cO$ an order in $\Da$. Then $\Gamma = \SL_3(\cO)$ is an arithmetic lattice in $\SL_6(\R)$ (for a non commutative ring $\cO$, we define $\SL_3(\cO)$ as the set of the invertible matrices in $\mathrm{M}_3(\cO)$ with the reduced norm $1$, cf. \cite[Section~2.3.1]{PR}). 
If $\Da$ splits over $\Q$, then $\Da = \mathrm{M}_2(\Q)$ and $\cO$ is, say, $\mathrm{M}_2(\Z)$ and we get $\Gamma = \SL_6(\Z)$ which has no bad primes. But in the general case there is a finite set of primes $p$ for which $\Da$ ramifies. For these primes $\Da(\Q_p)$ is a division algebra with residue degree $2$ and ramification index $2$ over $\Q_p$,
and one deduces that $\Da(\F_p)$ is also a division algebra, but as every division algebra over a finite field is a field, one sees that $\Da(\F_p) = \F_{p^2}$. It follows that the semisimple part of $\G(p)$ in this case is $\SL_3(p^2)$ instead of $\SL_6(p)$ which is the case when $p$ splits. This could cause a difficulty (note that $\gamma(\SL_3) > \gamma(\SL_6)$!). But fortunately this is not the case: a careful analysis of Prasad's volume formula shows that such bad primes increase the covolume of the given lattice, and hence decreases its contribution to the lattice growth.
\end{example}

\begin{step}
We now come back to the general case. Recall the function $h(K)$ which was defined in \cite[p.~108]{LN} by $h(K) = \frac{\log[\G(\F_q):K]}{\log |K^\diamond|},$ where $K$ is a subgroup of $\G(\F_q)$ and $K^\diamond$ denotes the maximal abelian quotient of $K$ whose order is coprime to $p = \mathrm{char(\F_q)}$. The main point of the proof of the upper bound in \cite{LN} is that the larger $h(K)$ is, the smaller is the contribution of $K$ to the subgroup growth, and the main technical result there is Theorem~4 (p.~109) asserting that 
$$\liminf_{q\to\infty} \min\{h(K) \mid K \le \G(\F_q)\} \ge R(\G).$$ 
We need to ensure that the same holds true if we allow the group $\G$ to vary, when the contribution to the lattice growth should be with respect to the covolume, so let us modify the definition of $h(K)$ to be
$$ h_\G(K) := \frac{\log([G':K] \lambda(q))}{\log |K^\diamond|}.$$
Here the group $G' = \G'(\F_{q'})$ is determined by the structure of $\G$ over the corresponding place of $k$. It is equal to $\G(\F_q)$ for the good primes and for the bad primes is defined depending on the type (see Steps~9--11 below). The factor $\lambda(q)$ compensates for the covolume increase contributed by the bad primes: for almost all places $v\in V_f(k)$ the group $\G$ is quasi-split over $k_v$ and $\Pa_v$ is a hyperspecial parahoric subgroup --- for all these places we have $\lambda(q) = 1$; now for the bad primes the $\lambda$-factor is determined by the ratio of the volume of a hyperspecial parahoric subgroup in a quasi-spit group of the same type as $\G$ over $k_v$ and the volume of $\Pa_v$ with respect to the Haar measure from \cite[Sections 1.3, 2.1]{Pr}.
It follows from Prasad's formula that for the principal arithmetic lattices the factors are given by $\lambda(q) = |\G_0(\F_q)|/|G'|$, where $\G_0$ is a simply connected quasisimple group of the same Lie type as $\G$ defined over the finite field $\F_q$.

We have to consider $h_\G(K)$ for the three types of bad primes defined above. As in \cite{LN} we will restrict to the places of $k$ with the residue characteristic $>3$. Since the degree of $k$ is bounded this assumption will not affect the growth function. 
\end{step}

\begin{step}
\emph{Type} (i): The group $\G$ splits over the maximal unramified extension $\hat{k}_v$ of $k_v$ and $\Pa_v$ is not a hyperspecial parahoric subgroup. Then assuming as in Section~\ref{sec:ar} that $\Pa_v = \G_v(\cO_v)$, we have $\G_0 = \G$, $G' = \overline{\G}_v(\F_q) \le \G(\F_q)$ and $\lambda(q) = [\G(\F_q):G']$, and therefore
\begin{equation}\label{eq_typei}
h_\G(K) = \frac{\log([G':K] \cdot [\G(\F_q):G'] )}{\log |K^\diamond|} = \frac{\log[\G(\F_q):K]}{\log |K^\diamond|} = h(K).
\end{equation}
Hence the problem reduces to the two other types of primes. 
\end{step}

\begin{step}
\emph{Type} (ii): Here the group $\G$ is not quasi-split over $k_v$ and splits over an unramified extension of $k_v$. We have $\G'(\F_{q'}) = \G(\F_{q})$, where $\F_{q}$ is the residue field of $k_v$ (in our example $q = p$ and $q' = p^2$). In the notation of \cite[Section~2.2]{Pr} (see also \cite[Section~4.3]{BL2}) the $\lambda$-factor is given by 
$$ \lambda(q) = \frac{q^{(\dim\overline{\mathrm{M}}_v + \dim\overline{\mathcal{M}}_v)/2}}{\# \overline{\mathrm{M}}_v(\F_q)} \cdot \frac{\# \overline{\mathcal{M}}_v(\F_q)}{q^{\dim\overline{\mathcal{M}}_v}}.$$
The reductive group $\overline{\mathcal{M}}_v$ associated to the quasi-split inner form of $\G$ is in fact absolutely quasi-simple (see \cite[Section 2.5]{Pr}), so its order over $\F_q$ can be obtained from the table of orders of finite groups of Lie type given in \cite[Table~1]{Ono}. 

In our Example~\ref{exmpl} we have $\overline{\mathcal{M}}_v = \SL_6$ and the group $\overline{\mathrm{M}}_v$ can be identified using the Bruhat--Tits theory; it is the product of the $1$-dimensional norm-$1$ torus $\mathrm{R}^{(1)}_{\F_{p^2}/\F_p}(\GL_1)$ and the semisimple group $\mathrm{R}_{\F_{p^2}/\F_p}(\SL_3)$ whose maximal $\F_p$-torus has dimension $4$ over $\F_p$ (we refer the reader to \cite{Tits} for a comprehensive survey of the Bruhat--Tits theory of reductive groups over nonarchimedean local fields).

It follows that the second factor in the product for $\lambda(q)$ is less than $1$ and bounded below by $c_1 = c_1(l) > 0$, where $l$ is the absolute rank of $\G$ (indeed, it tends to $1$ when $q \to \infty$). Moreover, we have $\# \overline{\mathrm{M}}_v(\F_q) \le (q+1)^{\dim \overline{\mathrm{M}}_v}$ (see \cite[Section~2.6 and Lemma~2.8]{Pr}). Therefore
$$ \lambda(q) \ge \frac{q^{(\dim\overline{\mathrm{M}}_v + \dim\overline{\mathcal{M}}_v)/2}}{ (q+1)^{\dim\overline{\mathrm{M}}_v}} \cdot c_1 \ge c_2 q^{(\dim\overline{\mathcal{M}}_v - \dim\overline{\mathrm{M}}_v)/2}.$$
Now, $\dim\overline{\mathcal{M}}_v = l + 2\#\Phi_+$ and $\dim\overline{\mathrm{M}}_v = l + 2\#\Phi'_+$, where $\Phi'_+$ is the set of positive roots of $\G'$ (note that the dimensions of the maximal $\F_q$-tori in $\overline{\mathcal{M}}_v$ and $\overline{\mathrm{M}}_v$ are equal). Hence we obtain $ \lambda(q) \ge c_2 q^{\#\Phi_+ - \#\Phi'_+}$, and
$$ h_\G(K) \ge \frac{\log([G':K] c_2 q^{\#\Phi_+ - \#\Phi'_+})}{\log |K^\diamond|}.$$
As in type (i), we can choose a suitable parahoric subgroup and with a computation similar to \eqref{eq_typei} reduce the problem to the case $G' = \G(\F_{q})$. Now a small modification of the proof of Proposition~3 in \cite{LN} implies that for large enough $q$ the function $h_\G(K)$ attains its minimum on the Borel subgroup of $G'$, for which we have $[G':K] \sim  q^{\#\Phi'_+}$ and $|K^\diamond| \sim q^l$. It follows that 
$$ \liminf_{q\to \infty} \min \{h_\G(K)\} \ge  \frac{\#\Phi_+}{l} = R(G).$$
It is important to note here that this bound is uniform in the lattice $\Lambda$, i.e. the rate of convergence depends only on the Lie type of $\G$. This indeed follows from the proofs of Propositions~2 and 3 in \cite{LN}.

\end{step}

\begin{step}
\emph{Type} (iii): Finally consider the case when $\G$ is not split over $\hat{k}_v$. Again as in type~(i) a computation similar to \eqref{eq_typei} allows us to choose a suitable parahoric subgroup
and assume $G' = \G'(\F_{q'})$, moreover, here $q' = q$ because the extension over which the quasi-split inner form of $\G$ splits is ramified over $k_v$.
Now Prasad's formula implies that for this type of primes we have $\lambda(q) \ge q^{s/2}$ with 
the constant $s = s(\G)$ defined in \cite[Section~0.4]{Pr} (cf. \cite[Section~4.4]{B}). From this one deduces:
\begin{align*}
h_G(K) \ge \frac{\log[G':K]}{\log |K^\diamond|} + \frac{(s/2)\log q}{\log |K^\diamond|};\\
\liminf_{q\to \infty} \min \{h_\G(K)\} \ge R(\G') + \frac{s}{2l'},
\end{align*}
where $l' = \mathrm{rank}(\G')$. 
Indeed, the second inequality follows immediately from \cite[Theorem~4]{LN} and the fact that $|K^\diamond| \lesssim q^{l'}$, 
and as in type~(ii) this bound is uniform in the lattice $\Lambda$.

If the quasi-split inner form of $\G$ does not split over an unramified extension of $k_v$ then $\G$ is an outer form over $k$. The types which admit outer forms together with the corresponding invariants are listed in Table~\ref{tbl2}. For the remaining types $\Bn_l$, $\Cn_l$, $\Gn$, $\Fn$, $\En_7$, $\En_8$ we have $R = l$, $l$, $3$, $6$, $9$, $15$, respectively. The table values for $s(\G)$ are provided by \cite[Section~0.4]{Pr}, the constant $R(\G)$ is easily computed from its definition, and the information about the type of $\G'$ follows from \cite[Section~4]{Tits}. We can now check that in all the cases we have $R(\G') + s/(2l') \ge R(\G)$. 

\renewcommand{\arraystretch}{1.2}
\begin{table}[ht] 
$$
\begin{array}{|l|l|l|l|}
\hline
\text{Type of }\G & s(\G) & R(\G) & \text{Type of }\G' \\
\hline
^2\An_l,\ l - \text{odd} & \frac12(l-1)(l+2) & \frac12(l+1) & \Cn_{\frac{l+1}2} \\
^2\An_l,\ l - \text{even} & \frac12l(l+3) & \frac12(l+1) & \Bn_{\frac{l}2} \\
^2\Dn_l  & 2l-1 & l-1 & \Bn_{l-1} \\
^3\Dn_4,\ ^6\Dn_4 & 7 & 3 & \Gn \\
^2\En_6 & 26 & 6 & \Fn \\
\hline
\end{array}
$$
\caption{}\label{tbl2}
\end{table}

With this modification at hand the rest of the proof in \cite{LN} indeed gives us the desired uniform upper bound. 

\end{step}

\begin{step}
We now move from $\Lambda'$ to $\Gamma$ of Step~3. To this end note that $\Lambda'(m)\lhd\Gamma$ and, again using \cite[Corollary~5.2]{GLP}, what we really need to count is $s_x(\Gamma/\Lambda'(m))$ where $m$ is a product of prime ideals in $k$.  So fix $m$ and let
$$N = \Lambda'/\Lambda'(m) \lhd L = \Gamma/\Lambda'(m), \quad Q = L/N.$$
By Corollary~\ref{level vs index} we can assume that $m \le x^{c_1+1}$ with $c_1 = c_1(H)$, and as the degree of the field $k$ is bounded it follows that $|N|$ is polynomially bounded in $m$ and hence in $x$.
By Proposition~\ref{cor:B}, $Q = \Gamma/\Lambda'$ is a finite group of order bounded by $c_2^{\log x /\log\log x}$ for a constant $c_2 = c_2(H)$. Together with the previous remark, it implies that $|L|$ is polynomially bounded in $x$.
Moreover, a prime $l$ can divide the order of $Q$ only if $l$ divides the order of the center of
the simply connected cover of the split form of $H$
or if it divides the order of the automorphism group of the local Dynkin diagram of $\G$, where $\G$ is an admissible group as in Section~\ref{sec:ar} (this can be deduced from \cite[Proposition~2.9]{BP} which is essentially due to Rohlfs \cite{Rohlfs}).
So only finitely many primes $l_1, \ldots , l_t$ can appear as the divisors (for a given $H$).

If $B$ is a subgroup of $L$ of index at most $x$, then $B\cap N$ is of index at most $x$ in $N$. As we already counted $s_x(N)$, we can assume that $B\cap N$ is given and estimate the number of possibilities for $B$.
For such a group $B$, $B/(B\cap N) \cong BN/N \le Q$, so the prime divisors of $|B/(B\cap N)|$ are among $\{l_1,\ldots, l_t\}$. It follows that $B$ is generated by $B\cap N$ and subgroups $B_1,\ldots,B_t$, where each $B_i$ is an $l_i$-Sylow subgroup of $B$. Each subgroup $B_i$ is contained in some $P_i$, where $P_i$ is an $l_i$-Sylow subgroup of $L$. So, in conclusion, such $B$ is generated by $B\cap N$ and by $\{P_1\cap B,\ldots,P_t\cap B\}$ and we assume that $B\cap N$ is fixed. For a fixed $i \in \{1, \ldots, t\}$, the number of $l_i$-Sylows of $L$ is at most $|L|$, which is polynomial in $x$, so the total number of such $t$-tuples $\{P_1,\ldots,P_t\}$ is polynomial in $x$ (as $t$ is bounded, depending only on $H$). Given such a $t$-tuple $\{P_1,\ldots,P_t\}$, $B$ as above is determined by $\{B\cap P_i \mid i = 1, \ldots, t\}$. Now, each $B\cap P_i$ contains
$B\cap N\cap P_i = B\cap P_i \cap N$, i.e., the number of possibilities for $B\cap P_i$ is bounded by the number of subgroups of
$P_i/(P_i\cap N) \cong P_iN/N$. The latter group $P_iN/N$ is a subgroup of $Q$, whose order is at most $c_3^{\log x /\log\log x}$. Hence the total number of subgroups of $Q$, and in particular those of the form $P_iN/N$, is bounded by $c_4^{(\log x /\log\log x)^2}$ (as the group of size $n$ has at most $n^{\log_2 n}$ subgroups, cf. \cite[Lemma~1.2.2]{LS}). So, in summary,
%
%
%
%
it follows that the total number of possible subgroups $B$ is bounded by \linebreak $s_x(\Lambda) x^{c\log x/(\log\log x)^2}$ with the constant $c$ depending only on $H$. The second factor is asymptotically smaller than $s_x(\Lambda)$, and we conclude that
$$
\limsup_{x\to\infty} \frac{\log \rnu}{(\log x)^2 / \log\log x} =
\limsup_{x\to\infty} \frac{\log s_x(\Lambda)}{(\log x)^2 / \log\log x}
\le \gamma(H).
$$
This finishes the proof of Theorem \ref{theorem}. \qed
\end{step}

\section{The GRH, rank of class groups and counting lattices}
\label{section:semisimple}\label{sec:semi_lower_thm2}\label{sec:semi_upper_thm2}\label{sec:equiv conj}

In this section we are going to discuss how to extend Theorem~\ref{theorem} to other simple and also semisimple Lie groups and its relation to some number theoretic conjectures.

We begin with the lower bound. First let us note that the lower bound in Theorem~\ref{theorem} is true unconditionally for every (non-compact) simple Lie group, including the non $2$-generic ones. Let us discuss its extension to semisimple Lie groups. Given such a group $H$, it is natural to consider only irreducible lattices in $H$, so from now on, $\mathrm{L}_H(x)$ denotes the number of conjugacy classes of \emph{irreducible} lattices in $H$ of covolume at most $x$. The same refers to $\ru$, $\rnu$ and other notations from Section~1. We recall (see Section~\ref{sec:ar}) that $H$ contains irreducible lattices only if it is isotypic, and that we can assume that $H = (\prod_{j=1}^{a}\G_j(\R) \times \G_{a+1}(\C)^{b})^o$ for some absolutely simple $\R$-groups $\G_j$, $j = 1, \dots, a+1$. Recall also that in contrast with the case of the uniform lattices, the isotypic condition is not sufficient for $H$ having a non-uniform irreducible lattice (see \cite[Example~18.7.7]{WM}).

In the proof of the lower bound of Theorem \ref{theorem} we showed that if $H$ is a high
rank simple Lie group then
\begin{equation*}
\liminf_{x\to\infty} \frac{\log \rnu}{(\log x)^2 / \log\log x} \ge
\gamma(H).
\end{equation*}
We did so by presenting an arithmetic lattice $\Gamma$ in $H$ which is defined
over $k$, $[k:\Q] \le 2$ and appealing to~\cite{GLP} where a lower bound is
given on the number of congruence subgroups of such a lattice $\Gamma$. The
results in~\cite{GLP}, in the most general form, rely on the generalized Riemann hypothesis (GRH).
To prove the theorem for simple Lie groups the GRH is not needed since $[k:\Q]\le 2$ and the Bombieri--Vinogradov Riemann
hypothesis on the average suffices --- see \cite{GLP}. The same is true if
\begin{itemize}
\item[$(7.1)$] \textit{$k$ is a number field contained in a Galois field $E$ over $\Q$,
such that $\Gal(E/\Q)$ has an abelian subgroup of index $4$}
\end{itemize}
(see \cite{GLP}, \cite{LN}).

Unfortunately, not every semisimple group $H$ admits a non-uniform irreducible
lattice defined over a field $k$ which satisfies $(7.1)$ (even if $H$ admits
non-uniform irreducible lattices). As we now show, $H =\SL_3(\R)^3\times\SL_3(\C)^2$
provides an example of such a group.

First note that if $k$ satisfies $(7.1)$ then $\Gal(E/\Q)$ is solvable. To define
a non-uniform lattice in $H = \SL_3(\R)^3\times\SL_3(\C)^2$ we need a number
field $k$ with $r_1 = 3$ real and $r_2 = 2$ complex places, so $[k:\Q] = r_1 +
2r_2 = 7$ is an odd prime. The following lemma, which is due to Peter M\"uller
and was communicated to us by Moshe Jarden, implies that condition $(7.1)$ never
holds for such fields $k$:
\begin{lemma}\label{lemma_81}
If $r_1 + 2r_2 = l$ is an odd prime and $1<r_1<l$, then there is no number field $k$ with
\begin{itemize}
\item[(1)] exactly $r_1$ real and $r_2$ complex places;
\item[(2)] solvable $\Gal(E/\Q)$, where $E$ is the Galois closure of $k$
over $\Q$.
\end{itemize}
\end{lemma}
\begin{proof} Assume $k$ is a field of degree $l$ which satisfies conditions
(1) and (2). Choose a primitive element $x$ for $k/\Q$ and let $f = {\rm irr}
(x, \Q)$ be its irreducible polynomial over $\Q$ (so $k = \Q[x]/(f(x))$). Let
$X$ denote the set of roots of $f$ in $E$, so by the assumption $\# X = l$ and
$\#(X\cap\R) = r_1$. As $1<r_1<l$, there is at least one real root and not all
roots are real. The group $G = \Gal(E/\Q)$ acts faithfully and transitively on
$X$. Since $\# X = l$ is a prime, $G$ is primitive group, and it is solvable by the assumption.
Let $N$ be a minimal normal subgroup of $G$. 
Then $|N| = l$, and G is isomorphic as a permutation group to a subgroup of the affine general linear group ${\rm AGL} (1, \F_l)$ acting on $\F_l$ (see \cite[Lemma~21.7.2]{FJ}).
Thus every element
of $G\smallsetminus N$ fixes exactly one element of $X$. In particular, if $\tau\in
G$ is an involution, then $\tau\not\in N$ (since $|N| = l$ and $l>2$). So
$\tau$ fixes exactly one element of $X$. This applies to all real involutions
in $G$, hence $\#(X\cap\R) = 1$ and so $r_1 = 1$, which is a contradiction.
\end{proof}

It follows that in order to extend our proof of the lower bound in Theorem \ref{theorem} to semisimple groups $H$ we need to assume the GRH.
This is the only obstacle, if we assume the GRH, we can carry on the proof of the lower bound in Section~\ref{sec:thm2 lower}, without any change, for every semisimple Lie group $H$, provided $H$ has a non-uniform irreducible lattice. In summary, we have:

\begin{thm}\label{thm:semisimple1}
Let $H$ be a semisimple Lie group of real rank at least $2$ which contains a non-uniform irreducible lattice. Then, assuming the GRH, we have
$$\liminf_{x\to\infty} \frac{\log \rnu}{(\log x)^2 / \log\log x} \ge \gamma(H),$$
where $\gamma(H)$ is defined by (\ref{eq_11}) applied to (any) of the simple factors of $H$.
\end{thm}

Let us move to the upper bound. Here there is an issue to discuss even for the case of general (i.e. not $2$-generic) simple Lie group.
The assumption imposed on the type of $H$ in Theorem~\ref{theorem} to be $2$-generic is used only for the bound of the order of the quotient group $Q$, which is provided by Proposition~\ref{cor:B}(ii) whose proof in \cite{B} uses Gauss's Theorem~\ref{thm:Gauss}. By Remark \ref{rem_43}, this argument would work for every simple Lie group, and even for general semisimple groups, if we would know Conjecture \ref{conj_h}.
We will now show that a partial converse is also true, namely, we will prove that if Theorem~\ref{theorem} holds for semisimple Lie groups of the form $\SL_n(\R)^{r_1}\times\SL_n(\C)^{r_2}$, then Conjecture~\ref{conj_h'} is true for all number fields $k$ with $r_1$ real and $r_2$ complex places and any odd prime $l$.

We first establish a supplementary result:
\begin{prop}\label{lemma_Cl_n}
Let $k$ be a number field, $H = \SL_n(k\otimes_\Q \R)$, $\Lambda = \SL_n(\cO)$ and $Q = N_H(\Lambda)/\Lambda$.
 If $n$ is odd then there exists an epimorphism
 $$\psi:Q \to \mathrm{Cl}_n(k).$$
\end{prop}


\begin{proof} The proof is based on some well known results in ring theory.

Let $A = \operatorname{M}_n(\cO)$. Recall that by a theorem of Rosenberg and Zelinsky \cite{RZ}, we have $$\mathrm{Out}(A)\cong \mathrm{Cl}_n(k).$$
An element $\gamma\in N_H(\Lambda)$ induces an automorphism $\theta$ of the matrix algebra $\operatorname{M}_n(\C)$ and, moreover, if $x \in A$, then $\theta(x) = \gamma^{-1}x\gamma \in A$ (this follows
from the fact that $\operatorname{M}_n(\cO)$ is the $\cO$-module spanned by $\SL_n(\cO)$). Therefore, we have a homomorphism \mbox{$f_1: N_H(\Lambda) \to \mathrm{Aut}(A)$} defined by $f_1(\gamma) = \theta$. Combining it with the natural epimorphism $f_2:\mathrm{Aut}(A) \to \mathrm{Out}(A)$, we get
\begin{align*}
N_H(\Lambda) &\stackrel{f_1}{\longrightarrow} \mathrm{Aut}(A) \stackrel{f_2}{\longrightarrow} \mathrm{Out}(A),\\
\Lambda      &\stackrel{f_1}{\longrightarrow} \mathrm{Inn}(A) \stackrel{f_2}{\longrightarrow} 1.
\end{align*}
Hence there is a homomorphism $\psi_0: N_H(\Lambda)/\Lambda \to \mathrm{Out}(A)$.

We want to show that $\psi_0$ is surjective. Any automorphism $\theta$ of $A$ extends to an automorphism of $A_k = A\otimes_\cO k$ which we denote by the same letter. By the Skolem--Noether theorem, every automorphism of $A_k$ is inner, so there exists an invertible element $u \in A_k = \operatorname{M}_n(k)$ such that for every $x\in A_k$, $\theta(x) = u^{-1}xu.$ Let $u_1 = u/(\det u)^{1/n}$. As $n$ is odd, we can choose the value of the root $(\det u)^{1/n}\in \R$ so that $u_1 \in \SL_n(k\otimes_\Q \R)$. Hence for every
$x\in \Lambda = \SL_n(\cO)$, $\theta(x) = u_1^{-1}xu_1 \in \Lambda$. This defines an element of $N_H(\Lambda)$ associated to $\theta$. As $\theta$ is an arbitrary automorphism of $A$, this shows that $f_1$, and hence $\psi_0$, is surjective. Note that the choice of the $n$-th root of $\det u$ which is involved in the construction of the preimage of $\theta$ does not allow us to claim that the inverse map would be a group homomorphism.

We proved that there exists an epimorphism $N_H(\Lambda)/\Lambda \to \mathrm{Out}(A)$ and it remains to apply the Rosenberg--Zelinsky theorem.
\end{proof}

\begin{rem}\label{rem_Cl_n}
The kernel of $\psi$ consists of those elements of $Q$ which come from the center of $H$, i.e. $\Ker(\psi) = Z(H)\Lambda$.
\end{rem}

%

We now show that if Theorem~\ref{theorem} holds for all $H$ of the form $\SL_n(\R)^{r_1}\times\SL_n(\C)^{r_2}$ for all $r_1$ and all $r_2$, then Conjecture~\ref{conj_h'} is true for all number fields $k$ and any odd prime $l$.

Let $k$ be a number field and let $\Lambda = \SL_l(\cO_k)$ be a corresponding principal arithmetic subgroup of $H = \SL_l(k\otimes_\Q \R)$. Then $\Lambda$ is a non-uniform lattice in $H$. The covolume of $\Lambda$ is given by Harder's Gauss--Bonnet formula (see \cite[Section~2.2]{Ha}):
\begin{equation}\label{eq_77}
\mu(H/\Lambda) = \D_k^{(l+1)(l-1)/2} \left(\prod_{j=1}^{l-1}\frac{j!}{(2\pi)^{j+1}}\right)^{[k:\Q]} \prod_{j=2}^l \zeta_k(j).
\end{equation}
(Here we use the normalization of $\mu$ from \cite{Pr}, the same result can be also obtained by Prasad's volume formula given there. The function $\zeta_k(j)$, which appears in the second product, is the Dedekind zeta function of the field $k$.) 

This implies that there exist positive constants $a, b, \delta$, which depend only on $H$, such that
\begin{equation}\label{eq_78}
a\D_k^{\delta} \le \mu(H/\Lambda) \le  b\D_k^{\delta}.
\end{equation}
Indeed, we can take $\delta = (l+1)(l-1)/2$ and then compute $a$ and $b$ from (\ref{eq_77}) using the basic fact that $1 \le \zeta_k(j) \le \zeta(j)^{[k:\Q]}$, where $\zeta(j)$ is the Riemann zeta function.

Now let us assume that Theorem~\ref{theorem} is true for $H = \SL_n(\R)^{r_1}\times\SL_n(\C)^{r_2}$ but Conjecture~\ref{conj_h'} is false for the fields of signature $(r_1, r_2)$ and some prime $l > 2$. So for any (large) number $r\in\R$, there exists such a field $k$ with $\rk_l(\mathrm{Cl}(k)) > r\log\D_k/\sqrt{\log\log \D_k}$, and thus $\mathrm{Cl}_l(k)$ has at least $l^{[\frac14 r^2\log^2 \D_k/(\log\log \D_k)]}$ subgroups (by \cite[Proposition~1.5.2]{LS}).

By Proposition~\ref{lemma_Cl_n}, $\Gamma/\Lambda$ admits an epimorphism onto $\mathrm{Cl}_l(k)$ where $\Lambda = \SL_l(\cO_k)$, $\Gamma = N_H(\Lambda)$ and $H = \SL_l(k\otimes_\Q \R)$. Let $x = \mu(H/\Lambda)$ and let $\Delta = Z(H)\Lambda$. We have:
$$s(\Gamma/\Delta) \ge s(\mathrm{Cl}_l(k)).$$
This gives a lower bound for the number of lattices between $\Gamma$ and $\Delta$.
If two such lattices $\Gamma_1$ and $\Gamma_2$ are conjugate by an element $h\in H$, then $h$ maps $\Lambda$ to $\Lambda$: Indeed, as $Q$ is nilpotent and $\Lambda = \SL_l(\cO_k)$, $l\ge 3$ is perfect (which can be checked, for example, by looking at the Steinberg relations defining it), $\Lambda$ is the unique maximal normal perfect subgroup of both $\Gamma_1$ and $\Gamma_2$ and hence $h$ takes it to itself, i.e. $h\in N_H(\Lambda)$. Now, as both $\Gamma_1$ and $\Gamma_2$ contain $\Delta$ and $N_H(\Lambda)/\Delta$ is abelian, we have $\Gamma_1 = \Gamma_2$. It follows that for $x = \mu(H/\Lambda)$, 
\begin{align*}
\rnu &\ge s(\Gamma/\Delta). 
\end{align*}
Hence we obtain
\begin{align*}
\rnu &\ge l^{[\frac14 r^2\log^2 \D_k/(\log\log \D_k)]};
\end{align*}
and using (\ref{eq_78}),
\begin{align*}
\rnu &\ge l^{cr^2\log^2 x/\log\log x};\\
\log \rnu &\ge (c\log l) r^2\log^2 x/\log\log x.
\end{align*}
As $r$ can be chosen arbitrary large this contradicts the assumption that Theorem~\ref{theorem} holds for $H$.

%

In summary, we obtain:

\begin{thm}\label{thm_equiv_conj}
\

\begin{itemize}
\item[(i)] If Conjecture \ref{conj_h} is true, then for any semisimple Lie group $H$ of real rank at least $2$ we have
$$\limsup_{x\to\infty} \frac{\log \rnu}{(\log x)^2 / \log\log x} \le \gamma(H).$$
\item[(ii)] If the bound in {\rm(i)} holds for $H = \SL_n(\R)^{r_1}\times\SL_n(\C)^{r_2}$ for every $n > 2$ and all $r_1$ and $r_2$, then Conjecture~\ref{conj_h'} is true for all number fields $k$ and all odd primes $l$.
\end{itemize}
\end{thm}

It is possible to elaborate further on the relation between Theorem~\ref{theorem} and Conjectures~3.3, in particular, Part~(ii) of Theorem~\ref{thm_equiv_conj} can be extended to other semisimple groups and prime $l = 2$, at least for totally imaginary fields. We shall not do it here. The main purpose of Theorem~\ref{thm_equiv_conj} is to highlight the intimate connection between counting non-uniform lattices and the class groups of fields.

\medskip

\noindent
{\bf Acknowledgements.} The authors would like to thank N.~Nikolov for sharing with us his know\-ledge about subgroup growth of lattices and valuable comments on the preliminary version of the paper. We also thank T.~Gelander, G.~Prasad and A.~Rapinchuk for helpful discussions. We are grateful to  M.~Jarden and P.~M\"uller for providing a proof of Lemma \ref{lemma_81} and we thank J.~Tsimerman for the reference to the paper by Brumer and Silverman.


\begin{thebibliography}{xxxx}

\bibitem[B]{B} M. Belolipetsky, Counting maximal arithmetic subgroups (with an appendix by J. Ellenberg and A. Venkatesh), {\em Duke Math. J.} {\bf 140} (2007) 1--33.

\bibitem[BGLS]{BGLS} M. Belolipetsky, T. Gelander, A. Lubotzky, A. Shalev, Counting arithmetic lattices and surfaces, {\em Ann. of Math. (2)} {\bf 172} (2010) 2197--2221.

\bibitem[BLi]{BLi} M. Belolipetsky, B. Linowitz, Counting isospectral manifolds, {\em  Adv. Math.} {\bf 321} (2017), 69--79. 

\bibitem[BLu]{BL2} M. Belolipetsky, A. Lubotzky,  Manifolds counting and class field towers, {\em Adv. Math.} {\bf 229} (2012) 3123--3146.

\bibitem[Bo1]{Bo1} A. Borel, Compact Clifford--Klein forms of symmetric spaces, {\em Topology} {\bf 2} (1963) 111--122.

\bibitem[Bo2]{Bo2} A. Borel, Commensurability classes and volumes of hyperbolic 3-manifolds, {\em Ann. Scuola Norm. Sup. Pisa (4)} {\bf 8} (1981) 1--33.

\bibitem[BH]{BH} A. Borel, G. Harder, Existence of discrete cocompact subgroups of reductive groups over local fields, {\em J. Reine Angew. Math.} {\bf 298} (1978) 53--64.

\bibitem[BP]{BP} A. Borel, G. Prasad, Finiteness theorems for discrete subgroups of bounded covolume in semi-simple groups, {\em Inst. Hautes \'Etudes Sci. Publ. Math.} {\bf 69} (1989) 119--171. Addendum: {\em ibid.,} {\bf 71} (1990) 173--177.

\bibitem[BS]{BS} Z. I. Borevich, I. P. Shafarevich, {\em Number Theory}, Acad. Press, 1966.


\bibitem[BrS]{BrS} A. Brumer, J. Silverman, The number of elliptic curves over $\Q$ with conductor $N$, {\em Manuscripta Math.} {\bf 91} (1996), 95--102.

\bibitem[BGLM]{BGLM} M. Burger, T. Gelander, A. Lubotzky, S. Mozes, Counting hyperbolic manifolds, {\em Geom. Funct. Anal.} {\bf 12} (2002) 1161--1173.

\bibitem[CR]{CR} V. I. Chernousov, A. A. Ryzhkov, On the classification of maximal arithmetic subgroups of simply connected groups,
(Russian) {\em Mat. Sb.} {\bf 188} (1997), no. 9, 127--156; translation in
{\em Sb. Math.} {\bf 188}~(1997) 1385--1413.

\bibitem[CL]{CL} H. Cohen, H. W. Lenstra, Jr., Heuristics on class groups of number fields, {\em Number theory, Noordwijkerhout 1983}, 33--62, Lecture Notes in Math. {\bf 1068}, Springer-Verlag, Berlin, 1984.

\bibitem[C]{Co} G. Cornell, Relative genus theory and the class group of $l$-extensions.
{\em Trans. Amer. Math. Soc.} {\bf 277}~(1983), 421--429.


\bibitem[EV]{EV}  J. Ellenberg, A. Venkatesh, Reflection principles and bounds for class group torsion, {\em Int. Math. Res. Not.} no.1, Art. ID rnm002 (2007).

\bibitem[FJ]{FJ} M. Fried, M. Jarden, {\em Field arithmetic, Second edition}, Springer-Verlag, Berlin, 2005.

\bibitem[G1]{G1} T. Gelander, Homotopy type and volume of locally symmetric manifolds, {\em Duke Math. J.} {\bf 124}~(2004), 459--515.

\bibitem[G2]{G2} T. Gelander, Non-compact arithmetic manifolds have simple homotopy type, {\em preprint} arXiv:math/0111261v1 [math.DG].

\bibitem[GLNP]{GLNP} D. Goldfeld, A. Lubotzky, N. Nikolov, L. Pyber, Counting primes, groups and manifolds, {\em Proc. of National Acad. of Sci.} {\bf 101}~(2004), 13428--13430.

\bibitem[GLP]{GLP} D. Goldfeld, A. Lubotzky, L. Pyber, Counting congruence subgroups, {\em Acta Math.} {\bf 193} (2004), 73--104.

\bibitem[H]{Ha} G. Harder, A Gauss-Bonnet formula for discrete arithmetically defined groups, {\em Ann. Sci. \'Ecole Norm. Sup. (4)} {\bf 4} (1971), 409--455.


\bibitem[L]{Lu} A. Lubotzky, Subgroup growth and congruence subgroups, {\em Invent Math.} {\bf 119} (1995), 267--295.

\bibitem[LN]{LN} A. Lubotzky, N. Nikolov, Subgroup growth of lattices in semisimple Lie groups, {\em Acta Math.} {\bf 193} (2004), 105--139.

\bibitem[LS]{LS} A. Lubotzky, D. Segal, {\em Subgroup growth}, Progr. Math. {\bf 212}, Birkh\"auser Verlag, Basel, 2003.

\bibitem[M]{Ma} G. A. Margulis, {\em Discrete subgroups of semisimple Lie groups}, Ergeb. Math. Grenzgeb. (3) {\bf 17}, Springer-Verlag, Berlin, 1991.


\bibitem[Mo1]{M1} D. Morris, Real representations of semisimple Lie algebras have $\Q$-forms, in {\em Algebraic groups and arithmetic}, Tata Inst. Fund. Res., Mumbai~(2004), 469--490.

\bibitem[Mo2]{WM} D. Witte Morris, {\em Introduction to arithmetic groups}, Deductive Press, 2015.

\bibitem[On]{Ono} T. Ono, On algebraic groups and discontinuous groups, {\em Nagoya Math. J.} {\bf 27} (1966) 279--322.

\bibitem[PlR]{PR} V. P. Platonov, A. S. Rapinchuk, {\em Algebraic groups and number theory}, Pure Appl. Math. {\bf 139}, Academic Press, Boston, 1994.

\bibitem[P]{Pr} G. Prasad, Volumes of $S$-arithmetic quotients of semi-simple groups, {\em Inst. Hautes \'Etudes Sci. Publ. Math.} {\bf 69} (1989), 91--117.

\bibitem[PrR1]{PR96} G. Prasad, A. S. Rapinchuk, Computation of the metaplectic kernel, {\em Inst. Hautes \'Etudes Sci. Publ. Math.} {\bf 84} (1996) 91--187.

\bibitem[PrR2]{PR06} G. Prasad, A. S. Rapinchuk, On the existence of isotropic forms of semi-simple algebraic groups over number fields with prescribed local behavior, {\em Adv. Math.} {\bf 207} (2006), 646--660.

\bibitem[Ra]{Rag} M. S. Raghunathan, {\em Discrete Subgroups of Lie Groups}, Springer, New York, 1972.

\bibitem[Ro]{Rohlfs} J. Rohlfs, Die maximalen arithmetisch definierten Untergruppen zerfallender einfacher Gruppen, {\em Math. Ann.} {\bf 244} (1979), 219--231.

\bibitem[RZ]{RZ} A. Rosenberg, D. Zelinsky, Automorphisms of separable algebras, {\em Pacific J. Math.} {\bf 11} (1961), 1109--1117.

\bibitem[SG]{SG} A. Salehi Golsefidy, Counting lattices in simple Lie groups: the positive characteristic case, {\em Duke Math. J.} {\bf 161} (2012), 431--481. 

\bibitem[S]{S1} J.-P. Serre, Le probl\`eme des groupes de congruence pour $\SL_2$, {\em Ann. of Math. (2)} {\bf 92} (1970), 489--527.


\bibitem[T]{Tits} J. Tits, Reductive groups over local fields, Proc. Symp. Pure Math. {\bf 33} (1979), Part~I, 29--69.

\bibitem[W]{Wa} H. C. Wang, Topics on totally discontinuous groups, in {\em Symmetric spaces} (ed. by W. M. Boothby and G. Weiss), Marcel Dekker~(1972), 459--487.

\end{thebibliography}
\end{document}